\documentclass[11pt,reqno]{amsart}

\usepackage[margin=1in]{geometry}
\usepackage[T1]{fontenc}
\usepackage{lmodern}
\usepackage[expansion=false]{microtype}
\usepackage{amsmath,amssymb,amsthm,mathtools}
\usepackage{aliascnt}
\usepackage{enumitem}
\usepackage{needspace}
\usepackage{graphicx}
\usepackage{xcolor}
\usepackage{tikz,tikz-cd}
\usetikzlibrary{arrows.meta,calc,positioning}
\usepackage{hyperref}
\hypersetup{hidelinks}
\usepackage[nameinlink,capitalise,noabbrev]{cleveref}

\allowdisplaybreaks
\numberwithin{equation}{section}

\theoremstyle{plain}
\newtheorem{theorem}{Theorem}[section]
\newaliascnt{proposition}{theorem}
\newtheorem{proposition}[proposition]{Proposition}
\aliascntresetthe{proposition}
\newaliascnt{lemma}{theorem}
\newtheorem{lemma}[lemma]{Lemma}
\aliascntresetthe{lemma}
\newaliascnt{corollary}{theorem}
\newtheorem{corollary}[corollary]{Corollary}
\aliascntresetthe{corollary}
\theoremstyle{definition}
\newaliascnt{definition}{theorem}
\newtheorem{definition}[definition]{Definition}
\aliascntresetthe{definition}
\newaliascnt{example}{theorem}
\newtheorem{example}[example]{Example}
\aliascntresetthe{example}
\theoremstyle{remark}
\newaliascnt{remark}{theorem}
\newtheorem{remark}[remark]{Remark}
\aliascntresetthe{remark}

\crefname{theorem}{Theorem}{Theorems}
\crefname{proposition}{Proposition}{Propositions}
\crefname{lemma}{Lemma}{Lemmas}
\crefname{corollary}{Corollary}{Corollaries}
\crefname{definition}{Definition}{Definitions}
\crefname{example}{Example}{Examples}
\crefname{remark}{Remark}{Remarks}

\newcommand{\Nzero}{\mathbb N_0}
\newcommand{\R}{\mathbb R}
\newcommand{\C}{\mathbb C}
\newcommand{\D}{\mathbb D}
\newcommand{\Torus}{\mathbb T}
\newcommand{\Cplus}{\mathbb C_+}
\newcommand{\Ran}{\operatorname{Ran}}
\newcommand{\Ker}{\operatorname{Ker}}
\newcommand{\closspan}{\overline{\operatorname{span}}}
\newcommand{\dd}{\,\mathrm d}
\newcommand{\ip}[2]{\left\langle #1,#2\right\rangle}
\newcommand{\norm}[1]{\left\lVert #1\right\rVert}
\newcommand{\abs}[1]{\left\lvert #1\right\rvert}
\newcommand{\Log}{\operatorname{Log}}

\begin{document}

\title[Mesoscopic redundancy]{Mesoscopic redundancy for dynamical frames
generated by atomic singular model operators}

\author{Ilya Krishtal}
\address{School of Mathematical and Statistical Sciences,
Northern Illinois University, DeKalb, IL 60115, USA}
\email{ikrishtal@niu.edu}
\author{Javad Mashreghi}
\address{D\'epartement de math\'ematiques et de statistique,
Universit\'e Laval, Qu\'ebec, Qu\'ebec G1V 0A6, Canada}
\email{javad.mashreghi@mat.ulaval.ca}
\date{}

\hypersetup{
  pdftitle={Mesoscopic Redundancy for Dynamical Frames Generated by Atomic Singular Model Operators},
  pdfauthor={Ilya Krishtal and Javad Mashreghi},
  pdfsubject={Direct continuous-time characterization of fractional atomic model orbits by Kummer--Bessel transmutation and mesoscopic thickness},
  pdfkeywords={dynamical frame, fractional powers, Kummer function, transmutation operator, Bessel equation, model space, Kummer thickness, mesoscopic thickness}
}

\begin{abstract}
Consider the compressed shift $T=P_{K_\theta}M_z|_{K_\theta}$ on
$K_\theta=H^2\ominus\theta H^2$ associated with the atomic singular inner
function $\theta(z)=\exp(-a(\zeta+z)/(\zeta-z))$, $a>0$ and $|\zeta|=1$.
Let $k_0=P_{K_\theta}1$ and define real powers $T^s$ using a holomorphic
logarithm near the singleton spectrum of $T$. For any temporal
multisequence $\mathsf S=(s_k)_{k\in I}\subset[0,\infty)$ with uniformly
bounded numbers of samples in unit intervals, we prove that
$\{T^{s_k}k_0:k\in I\}$ is a frame if and only if $\mathsf S$ is
\emph{Kummer-thick}: the frequencies $\omega(s_k)=\sqrt{8s_k+4}$, weighted
by $1/\omega(s_k)$, have uniformly positive total mass in every sufficiently
distant interval of some fixed length. Equivalently, the occupied unit
cells fill a fixed positive proportion of every window
$[N-C\sqrt N,N+C\sqrt N]$ for some $C>0$ and all sufficiently large $N$.
The proof combines a Volterra transmutation into Bessel waves with the
Fourier--Bessel Logvinenko--Sereda theorem.
\end{abstract}


\maketitle
\enlargethispage{\baselineskip}

\section{Introduction}
  A frame of iterations in a separable Hilbert space $\mathcal H$  is a family
$\{T^nf_j:n\in\Nzero,1\le j\le m\}$ satisfying
\[
 A\norm h^2\le
 \sum_{j=1}^m\sum_{n\ge0}|\ip h{T^nf_j}|^2
 \le B\norm h^2,\qquad h\in\mathcal H,
\]
for some $0<A\le B<\infty$. Here $T\in\mathcal B(\mathcal H)$,
$f_1,\ldots,f_m\in\mathcal H$, and $\Nzero=\{0,1,2,\ldots\}$.
An indexed family is called Bessel if it satisfies the upper inequality.
Such systems are central in
dynamical sampling and evolving-data problems; see
\cite{ACKM26,ACMT17}.  Their structural theory includes vector-valued Hardy
models \cite{CMS23}, optimal finite-generator realizations \cite{ACNP26},
and recent results for multiplication operators and broader classes of
operator orbits \cite{AC26,GGP26}.  Carleson-type dynamical frames have
provided a particularly fertile source of redundancy questions
\cite{CHPS24,KMBDC26}; the block-diagonal theory in \cite{KMBDC26} uses
Jordan structure and natural density to obtain exact redundancy results.  The sectorial theory in \cite{KMMsectorial26} develops compact
localization and fixed-scale Beurling-density criteria for finitely
generated sectorial frames.
The atomic singular model considered here exhibits a different phenomenon:
the correct local scale grows like $\sqrt N$, and a direct Kummer--Bessel
transmutation makes this mesoscopic geometry explicit.

The present work may also be viewed as a temporal counterpart of the
space--time sampling problem for convolution flows studied in
\cite{AGHJKR21}.  There, uniform sub-Nyquist spatial sampling creates blind
spots, and stable recovery imposes quantitative restrictions on spatial gaps;
here, the observation vector is fixed, and we characterize the discrete
observation times for a highly non-normal model flow.

Let $H^2$ denote the Hardy space on the unit disk $\D$.  For an inner
function $\theta$, write
\[
 K_\theta=H^2\ominus\theta H^2,
 \qquad
 S_\theta=P_{K_\theta}M_z|_{K_\theta},
\]
where $P_{K_\theta}$ is the orthogonal projection and $M_z$ is multiplication
by $z$.  For $a>0$, put
\begin{equation}\label{eq:theta-a-intro}
 \theta_a(z)=\exp\!\left(-a\frac{1+z}{1-z}\right),\qquad
 K_a=K_{\theta_a},\qquad
 T_a=S_{\theta_a},\quad k_0=P_{K_a}1.
\end{equation}
The standard model-operator spectrum formula gives
$\sigma(T_a)=\{1\}$; see \cite{Nikolski02,Sarason94}.  Let $\Log$ be the
holomorphic logarithm near $1$ normalized by $\Log 1=0$, and define
\begin{equation}\label{eq:fractional-powers}
 T_a^s=\exp\!\bigl(s\Log T_a\bigr),\qquad s\in\R,
\end{equation}
using the holomorphic functional calculus.  These are the real powers used
throughout.  The reproducing-kernel notation $k_0^{\theta_a}$ is also
standard \cite{GMR16}; we suppress the superscript.
The integer orbit $\{T_a^nk_0\}_{n\ge0}=\{P_{K_a}z^n\}_{n\ge0}$ is
a Parseval frame.  We show that its natural temporal scale is not fixed:
near time $N$, stability is governed by windows of length comparable to
$\sqrt N$.

For a countable set $I$, let
$\mathsf S=(s_k)_{k\in I}\subset[0,\infty)$ be a locally finite temporal
multisequence and assume
\begin{equation}\label{eq:intro-multiplicity}
 M_1(\mathsf S):=\sup_{n\in\Nzero}\#\{k:s_k\in[n,n+1)\}<\infty.
\end{equation}
Define $\omega(s)=\sqrt{8s+4}$.  We call $\mathsf S$ \emph{Kummer-thick}
if there exist $L,\eta,R_0>0$ such that
\begin{equation}\label{eq:intro-main}
 \sum_{\abs{\omega(s_k)-R}\le L}\frac1{\omega(s_k)}\ge\eta,
 \qquad R\ge R_0.
\end{equation}
Our main result, \cref{thm:main}, states that
\[
 \{T_a^{s_k}k_0:k\in I\}\text{ is a frame}
 \quad\Longleftrightarrow\quad
 \mathsf S\text{ is Kummer-thick}.
\]
In temporal coordinates, this is equivalent to requiring a fixed positive
multiple of $\sqrt N$ samples, counted with multiplicity, in every window of
length comparable to $\sqrt N$.
Under \eqref{eq:intro-multiplicity}, it can also be expressed using integer occupied cells.  The criterion is independent of the atomic mass $a>0$.
It is also stable under uniformly bounded displacements in the radial
variable $\omega$, allowing temporal perturbations on the mesoscopic scale
when bounded unit-cell multiplicity is preserved; see \cref{cor:jitter}.

We use the standard notation
\[
 {}_1F_1(\alpha;\beta;z)
 =\sum_{n=0}^\infty
   \frac{(\alpha)_n}{(\beta)_n}\frac{z^n}{n!}
\]
for Kummer's confluent hypergeometric function (with
$\beta\notin\{0,-1,-2,\ldots\}$), where $(\alpha)_0=1$ and
$(\alpha)_n=\alpha(\alpha+1)\cdots(\alpha+n-1)$ is the Pochhammer symbol.
The key is an exact continuous-time reduction.  Put $b=\sqrt a$.  Under the
Cayley--Laplace and square-root changes of variables, $T_a^sk_0$ becomes
\[
 \psi_s(x)=2\sqrt x\,e^{-x^2}{}_1F_1(-s;1;2x^2),
 \qquad 0<x<b,
\]
which is the regular solution of a perturbed Bessel equation at the spectral parameter
$\omega(s)$.  If $J_0$ denotes the Bessel function of the first kind of
 zero order, the transmutation theory for perturbed Bessel equations developed by Kravchenko and Torba \cite{KT21} gives one Volterra operator $\mathcal W_b$, independent of $s$, such that
\(
 \psi_s(x)=\mathcal W_b\bigl(2\sqrt xJ_0(\omega(s)x)\bigr);
\)
see also
Holzleitner \cite{Holzleitner20}.  Since $\mathcal W_b$ is boundedly
invertible on $L^2(0,b)$, the fractional dynamical system is a frame exactly
when the associated Bessel waves form a frame.

The Bessel-wave problem is solved by discretizing the
necessary-and-sufficient Fourier--Bessel Logvinenko--Sereda theorem of
Ghobber and Jaming \cite{GJ13} on equal
$\rho\,\dd\rho$-mass cells.  The integer case has an additional Hardy-coordinate structure, recorded separately.

The remainder of the paper is structured as follows.
Section~2 recalls the exact integer geometry.  Section~3 obtains the
continuous Kummer realization, Section~4 gives the Bessel transmutation,
Section~5 compares the equivalent mesoscopic conditions, and Section~6
proves the fractional sampling theorem.  Integer refinements and other
frame generators are treated in Sections~7--8.

\section{Atomic model spaces and integer erasure geometry}
Integer times retain a coordinate structure that is absent for general
fractional samples.  We first isolate this additional Hardy-space geometry,
including the exact angle and coefficient-interpolation criteria for retained
subsystems of the canonical Parseval orbit.

For an inner function $\theta$, let
\[
 K_\theta=H^2\ominus\theta H^2,
 \quad S_\theta=P_{K_\theta}M_z|_{K_\theta}, \quad  k_0=P_{K_\theta}1.
 \]
We then have $S_\theta^nk_0=P_{K_\theta}z^n$.
The orthogonal projection of the monomial basis gives
\begin{equation}\label{eq:canonical-parseval}
 \sum_{n\ge0}\abs{\ip f{P_{K_\theta}z^n}}^2=\norm f^2,
 \qquad f\in K_\theta.
\end{equation}

For $E\subset\Nzero$, let $H_E^2=\closspan\{z^n:n\in E\}$ and denote by
$Q_E$ the coefficient projection onto $H_E^2$.  If $\Gamma$ is retained, write
$\Lambda=\Nzero\setminus\Gamma$.
Our first proposition records the exact integer angle and the interpolation criterion.

\begin{proposition}
\label{prop:exact-angle}
For a non-constant inner function $\theta$ and $\Gamma\subset\Nzero$, the
following are equivalent:
\begin{enumerate}[label=\textup{(\roman*)}]
\item $\{P_{K_\theta}z^n:n\in\Gamma\}$ is a frame for $K_\theta$;
\item $\norm{Q_\Lambda P_{K_\theta}}<1$;
\item $P_{\theta H^2}$ is bounded below on $H_\Lambda^2$;
\item $Q_\Lambda:\theta H^2\to H_\Lambda^2$ is onto.
\end{enumerate}
The optimal lower frame bound is
\begin{equation}\label{eq:optimal-angle}
 A_\Gamma=1-\norm{Q_\Lambda P_{K_\theta}}^2.
\end{equation}
\end{proposition}

\begin{proof}
For $f\in K_\theta$, \eqref{eq:canonical-parseval} gives
\[
 \sum_{n\in\Gamma}\abs{\ip f{P_{K_\theta}z^n}}^2
 =\norm f^2-\norm{Q_\Lambda f}^2.
\]
This proves the equivalence of (i) and (ii) as well as
\eqref{eq:optimal-angle}.  Let $M_\theta$ denote multiplication by $\theta$
in $H^2$.  The operator
$R_\Lambda=Q_\Lambda M_\theta:H^2\to H_\Lambda^2$ has adjoint
$M_\theta^*|_{H_\Lambda^2}$ and
$\norm{M_\theta^*h}=\norm{P_{\theta H^2}h}$. The surjectivity of an operator is
equivalent to its adjoint being bounded below, proving the equivalence of
(iii) and (iv).  Finally,
\[
 \norm{P_{\theta H^2}h}^2
 =\norm h^2-\norm{P_{K_\theta}h}^2,
 \qquad h\in H_\Lambda^2,
\]
and
\[
 (Q_\Lambda P_{K_\theta})^*=P_{K_\theta}Q_\Lambda,
 \qquad
 \norm{P_{K_\theta}Q_\Lambda}
 =\norm{P_{K_\theta}|_{H_\Lambda^2}}.
\]
Hence
$P_{\theta H^2}$ is bounded below on $H_\Lambda^2$ if and only if
$\norm{Q_\Lambda P_{K_\theta}}<1$, proving the equivalence of (ii) and (iii).
\end{proof}

\begin{remark}\label{prop:finite-erasure}
If $S_\theta$ is invertible, every finite subset of the canonical orbit may
be erased while preserving the frame property.
Indeed, every tail is an invertible image
\[
 \{P_{K_\theta}z^{N+n}:n\ge0\}=S_\theta^N
 \{P_{K_\theta}z^n:n\ge0\}
\]
of the full frame.  A finitely erased subsystem contains a full tail and
remains Bessel, hence is a frame.  A singular inner function has no zeros in
$\D$, so the standard model-operator spectrum formula gives
$0\notin\sigma(S_\theta)$; see \cite{Nikolski02,Sarason94}.
\end{remark}

\section{Continuous-time Kummer realization}

We now pass from the discrete Hardy-space orbit to a continuous-time
realization. The Cayley--Laplace model identifies every fractional orbit
vector with a continuous-degree Kummer function and reveals the natural
spectral parameter $\omega(s)=\sqrt{8s+4}$.

It suffices to treat the atom at $1$; other boundary points follow by
rotation. Let $\Cplus=\{z:\Re z>0\}$ and normalize its Hardy-space norm by
\[
 \norm F_{H^2(\Cplus)}^2
 =\sup_{\sigma>0}\frac1{2\pi}
   \int_{\R}|F(\sigma+iy)|^2\,\dd y.
\]
With the usual coefficient norm on $H^2(\D)$, the Cayley map
\begin{equation}\label{eq:cayley}
 (\mathcal C f)(z)=\frac{\sqrt2}{z+1}f\!\left(\frac{z-1}{z+1}\right)
\end{equation}
is unitary from $H^2(\D)$ to $H^2(\Cplus)$, while the Laplace transform
\begin{equation}\label{eq:laplace}
 (\mathcal Lg)(z)=\int_0^\infty e^{-zt}g(t)\,\dd t
\end{equation}
is unitary from $L^2(0,\infty)$ onto $H^2(\Cplus)$;
see \cite{GMR23,Widder46}.  Multiplication by $e^{-az}$
corresponds to right translation by $a$, and hence
\begin{equation}\label{eq:Ka-L2}
 K_a\simeq L^2(0,a).
\end{equation}
More precisely, we use the unitary identification
\begin{equation}\label{eq:U-def}
 \mathcal U=\mathcal L^{-1}\mathcal C|_{K_a}:K_a\longrightarrow L^2(0,a).
\end{equation}

In $L^2(0,a)$ define
\begin{equation}\label{eq:volterra-a}
 (\mathcal V_ag)(t)=\int_0^t e^{-(t-u)}g(u)\,\dd u.
\end{equation}
Multiplication by $(z+1)^{-1}$ is the Laplace transform of convolution with
$e^{-t}$, so the atomic model operator $T_a$ corresponds to
\begin{equation}\label{eq:B-a}
 B_a=I-2\mathcal V_a.
\end{equation}
The vector $k_0$ corresponds to $e_0(t)=\sqrt2e^{-t}$.

Recall that ${}_1F_1(\alpha;\beta;z)$ denotes Kummer's confluent
hypergeometric function, as defined in the introduction.
The following proposition, illustrated by \cref{fig:kummer-commutative},
provides an $L^2(0,a)$ realization of the fractional atomic orbit.

\begin{proposition}
\label{prop:kummer-orbit}
For every $s\ge0$, the vector $T_a^sk_0$ corresponds under
\eqref{eq:Ka-L2} to
\begin{equation}\label{eq:ell-s}
 \ell_s(t)=\sqrt2e^{-t}\,{}_1F_1(-s;1;2t),
 \qquad 0<t<a.
\end{equation}
After the unitary substitution
\begin{equation}\label{eq:V-square}
 (Vg)(x)=\sqrt{2x}\,g(x^2),
 \qquad V:L^2(0,a)\to L^2(0,b),\quad b=\sqrt a,
\end{equation}
one obtains
\begin{equation}\label{eq:psi-s}
 \psi_s(x)=2\sqrt x\,e^{-x^2}{}_1F_1(-s;1;2x^2).
\end{equation}
This is the regular solution at $0$ of
\begin{equation}\label{eq:continuous-oscillator}
 -\psi_s''(x)+\left(4x^2-\frac1{4x^2}\right)\psi_s(x)
 =\omega(s)^2\psi_s(x),
 \qquad \omega(s)=\sqrt{8s+4},
\end{equation}
normalized by $\psi_s(x)\sim2\sqrt x$ as $x\downarrow0$.
\end{proposition}

\begin{figure}[t]
\centering
\begin{minipage}{\linewidth}
\centering\small
\begin{tikzcd}[column sep=large,row sep=large]
K_a
  \arrow[r,"T_a^s"]
  \arrow[d,"\mathcal U"']
&
K_a
  \arrow[d,"\mathcal U"]
\\
L^2(0,a)
  \arrow[r,"B_a^s"]
  \arrow[d,"V"']
&
L^2(0,a)
  \arrow[d,"V"]
\\
L^2(0,b)
  \arrow[r,"\widetilde B_a^s"']
&
L^2(0,b)
\end{tikzcd}

\vspace{0.5em}

\[
\begin{aligned}
(\mathcal U k_0)(t)
   &= e_0(t)=\sqrt2\,e^{-t},\\
(\mathcal U T_a^s k_0)(t)
   &=(B_a^s e_0)(t)
     =\ell_s(t)
     =\sqrt2\,e^{-t}\,
       {}_1F_1(-s;1;2t),\\
(V\ell_s)(x)
   &=\psi_s(x)
     =2\sqrt{x}\,e^{-x^2}
       {}_1F_1(-s;1;2x^2),
\qquad b=\sqrt a .
\end{aligned}
\]
\end{minipage}
\caption{The continuous Kummer realization of the fractional atomic
orbit.  Here $B_a=\mathcal U T_a\mathcal U^{-1}=I-2\mathcal V_a$ and
$\widetilde B_a^s=VB_a^sV^{-1}$; in particular,
$\widetilde B_a^s(Ve_0)=\psi_s$}
\label{fig:kummer-commutative}
\end{figure}

\begin{proof}
The holomorphic functional calculus is preserved by $\mathcal U$, so
$T_a^s$ corresponds to $B_a^s$.  The operator $\mathcal V_a$ is
quasinilpotent, and the logarithm fixed in \eqref{eq:fractional-powers}
therefore gives $B_a^s=(I-2\mathcal V_a)^s$.  The Taylor series of
$(1-2z)^s$ at $0$ converges in operator norm at $\mathcal V_a$: its scalar
coefficients have exponential growth rate at most $2$, whereas
$\norm{\mathcal V_a^n}^{1/n}\to0$.  Since
\[
 \mathcal V_a^n(e^{-t})=\frac{t^n}{n!}e^{-t},
\]
we obtain
\begin{align*}
 B_a^se_0
 &=\sqrt2e^{-t}\sum_{n=0}^\infty
    \binom{s}{n}\frac{(-2t)^n}{n!}\\
 &=\sqrt2e^{-t}\sum_{n=0}^\infty
    \frac{(-s)_n}{(n!)^2}(2t)^n
 =\ell_s(t).
\end{align*}
This proves \eqref{eq:ell-s}; the series is locally uniform in $(s,t)$.
The formula \eqref{eq:psi-s} follows from \eqref{eq:V-square}.

Kummer's equation, in the normalization of \cite[\S13.2]{DLMF}, is
\[
 zM''+(1-z)M'+sM=0,
 \qquad M(z)={}_1F_1(-s;1;z).
\]
Substituting $z=2x^2$ and multiplying by $2\sqrt x\,e^{-x^2}$ gives \eqref{eq:continuous-oscillator}.  The power series at the
origin gives the stated regular normalization.
\end{proof}

For integer $s=n$, ${}_1F_1(-n;1;2t)=L_n(2t)$, where $L_n$ is the
$n$th Laguerre polynomial. Thus, \eqref{eq:ell-s} reduces to the classical
normalized Laguerre function.

\section{Volterra transmutation to Bessel waves}

The Kummer realization produces regular solutions of a perturbed Bessel
equation.  To obtain a tractable sampling problem, we remove the oscillator
potential by a single Volterra transmutation that is independent of the
temporal parameter.

For $\omega>0$, set
\begin{equation}\label{eq:u-omega}
 u_\omega(x)=2\sqrt x\,J_0(\omega x),
 \qquad 0<x<b,
\end{equation}
where $J_0$ is the Bessel function of the first kind of
 zero order.
It is the regular solution of
\begin{equation}\label{eq:unperturbed}
 -u''-\frac1{4x^2}u=\omega^2u,
 \qquad u(x)\sim2\sqrt x.
\end{equation}

The following key transmutation theorem relates the function families $\{\psi_s\}$  and $\{u_{\omega(s)}\}$, $\omega(s)=\sqrt{8s+4}$.

\begin{theorem}
\label{thm:transmutation}
There is a continuous kernel $K_b(x,t)$ on
$0\le t\le x\le b$ such that
\begin{equation}\label{eq:W-b}
 (\mathcal W_bf)(x)=f(x)+\int_0^xK_b(x,t)f(t)\,\dd t
\end{equation}
defines a boundedly invertible operator on $L^2(0,b)$ and
\begin{equation}\label{eq:transmutation-identity}
 \psi_s=\mathcal W_bu_{\omega(s)},
 \qquad s\ge0.
\end{equation}
Consequently, for every countable temporal multisequence $\mathsf S=(s_k)_{k\in I}$,
\begin{equation}\label{eq:frame-transmutation-equivalence}
 \{T_a^{s_k}k_0\}_{k\in I}\text{ is a frame for }K_a
 \quad\Longleftrightarrow\quad
 \{u_{\omega(s_k)}\}_{k\in I}\text{ is a frame for }L^2(0,b).
\end{equation}
\end{theorem}

\begin{proof}
We use the Kravchenko--Torba transmutation theorem for perturbed Bessel
equations \cite[Section~2, equation~(2.3)]{KT21}, with
$\ell=-\tfrac12$ and $q(x)=4x^2$. The cited result includes this endpoint,
and $q$ is continuous on $[0,b]$. In particular, the logarithmic
integrability condition
$\int_0^b(1+|\log(x/b)|)|q(x)|\,\dd x<\infty$ is satisfied.
We deduce that the
regular solution of
\[
 -y''+\left(-\frac1{4x^2}+4x^2\right)y=\omega^2y
\]
is obtained from the regular solution of the unperturbed equation \eqref{eq:unperturbed} by a
Volterra operator with a continuous kernel, independent of $\omega$.  Their
reference solution is
\[
 b_{-1/2}(\omega x)=\sqrt{\frac{\pi\omega x}{2}}J_0(\omega x).
\]
Multiplying both regular solutions by $2\sqrt{2/(\pi\omega)}$ gives the
normalization $2\sqrt x$ at the origin.  Uniqueness of the regular solution
and \cref{prop:kummer-orbit} then give
\eqref{eq:transmutation-identity}.  Related transformation-operator results
for spherical Schr\"odinger equations appear in \cite{Holzleitner20}.

The integral part of $\mathcal W_b$ is Volterra with a continuous kernel.  If
$M=\max\limits_{0\le t\le x\le b}|K_b(x,t)|$, its $n$th power has kernel bounded by
$\frac{M^n(x-t)^{n-1}}{(n-1)!}$.  Hence, the resolvent Neumann series converges in
the operator norm and $\mathcal W_b$ is boundedly invertible on $L^2(0,b)$.
An invertible operator applied to all vectors in a family preserves the
frame property, proving \eqref{eq:frame-transmutation-equivalence}.
\end{proof}

\section{Kummer thickness and mesoscopic geometry}

The transmutation reduces the frame problem to sampling Bessel waves at 
frequencies $\omega(s_k)$.  We next identify the corresponding radial
thickness condition and translate it into a variable-scale temporal
geometry of windows of length comparable to $\sqrt N$.

\begin{definition}
\label{def:temporal-set}
Let $\mathsf S=(s_k)_{k\in I}\subset[0,\infty)$ be a locally finite temporal
multisequence with repetitions counted.  Define
\[
 M_1(\mathsf S)=\sup_{n\in\Nzero}\#\{k:s_k\in[n,n+1)\},
 \qquad
 \Gamma_{\mathsf S}=\{\lfloor s_k\rfloor:k\in I\}.
\]
We assume throughout that $M_1(\mathsf S)<\infty$.  Put
$\omega_k=\omega(s_k)=\sqrt{8s_k+4}$.  For $\xi\ge0$, let $\delta_\xi$
denote the unit point mass at $\xi$, and define
\begin{equation}\label{eq:nu-S}
 \nu_{\mathsf S}=\sum_{k\in I}\frac1{\omega_k}\,\delta_{\omega_k}.
\end{equation}
The multisequence $\mathsf S$ is called \emph{Kummer-thick} if there are
$L,\eta,R_0>0$ such that
\begin{equation}\label{eq:weighted-thickness-direct}
 \nu_{\mathsf S}([R-L,R+L])\ge\eta,
 \qquad R\ge R_0.
\end{equation}
\end{definition}

The weights in \eqref{eq:nu-S} reflect the sizes of the Bessel waves.
If $J_1$ denotes the Bessel function of the first kind of order one, then
\begin{equation}\label{eq:bessel-wave-norm}
 \norm{u_\omega}_{L^2(0,b)}^2
 =2b^2\bigl(J_0(\omega b)^2+J_1(\omega b)^2\bigr)
 \sim\frac{4b}{\pi\omega},\qquad\omega\to\infty.
\end{equation}
The identity follows by integrating $4xJ_0(\omega x)^2$, and the
asymptotic follows from the large-argument expansions of $J_0$ and $J_1$;
see \cite[\S\S10.22(i), 10.17(i)]{DLMF}.
Thus, $1/\omega$ agrees, up to a constant, with the asymptotic squared
norm of the corresponding wave.

A set $\Gamma\subset\Nzero$ is called $\sqrt N$-thick if there are
$C,c,N_0>0$ such that, for every real $N\ge N_0$,
\begin{equation}\label{eq:sqrt-thick}
 \#\bigl(\Gamma\cap[N-C\sqrt N,N+C\sqrt N]\bigr)
 \ge c\sqrt N,
 \qquad N\ge N_0.
\end{equation}

Figure~\ref{fig:gamma-sets} contrasts the two relevant local counting behaviors.  If \(\Gamma_{\mathrm{per}}\) is a nonempty periodic subset of \(\mathbb N_0\), then it has some positive density \(\delta\), and
\[
   \#\bigl(
      \Gamma_{\mathrm{per}}
      \cap[N-C\sqrt N,N+C\sqrt N]
   \bigr)
   =
   2C\delta\sqrt N+O(1).
\]
Hence \(\Gamma_{\mathrm{per}}\) is \(\sqrt N\)-thick. By contrast, for
\(\Gamma_{\mathrm{sq}}=\{k^2:k\in\mathbb N_0\}\), the spacing between consecutive points near \(N\) is
\[
   (k+1)^2-k^2=2k+1\asymp\sqrt N.
\]
Consequently, a window of length $O(\sqrt N)$ contains only $O(1)$ squares,
rather than the required positive multiple of $\sqrt N$.  Thus
$\Gamma_{\mathrm{sq}}$ is not $\sqrt N$-thick.

\begin{figure}[t]
\centering
\resizebox{0.95\linewidth}{!}{%
\begin{tikzpicture}[x=0.095cm,y=1cm,>=Latex]
  \def\leftwin{48}\def\rightwin{80}
  \fill[black!8] (\leftwin,1.25) rectangle (\rightwin,2.05);
  \fill[black!8] (\leftwin,-0.05) rectangle (\rightwin,0.75);
  \draw[->] (-1,1.65) -- (103,1.65) node[right] {$n$};
  \draw[->] (-1,0.35) -- (103,0.35) node[right] {$n$};
  \foreach \n in {0,...,25}{\fill ({4*\n},1.65) circle (1.55pt);
    \pgfmathtruncatemacro{\m}{4*\n+1}
    \ifnum\m<101\relax\fill (\m,1.65) circle (1.55pt);\fi}
  \foreach \k in {0,...,10}{\fill ({\k*\k},0.35) circle (1.55pt);}
  \foreach \x in {0,16,32,48,64,80,96}{
    \draw (\x,1.59)--(\x,1.71) node[above=3pt,scale=.75] {\x};
    \draw (\x,0.29)--(\x,0.41) node[below=3pt,scale=.75] {\x};}
  \node[anchor=east] at (-4,1.65) {$\Gamma_{\rm per}$};
  \node[anchor=east] at (-4,0.35) {$\Gamma_{\rm sq}$};
  \draw[<->,thick] (\leftwin,1.25)--(\rightwin,1.25)
       node[midway,fill=white,inner sep=1.5pt] {$4\sqrt{64}$};
\end{tikzpicture}%
}
\caption{Periodic occupancy satisfies the mesoscopic condition, whereas the
quadratic set $\{k^2\}$ does not}
\label{fig:gamma-sets}
\end{figure}
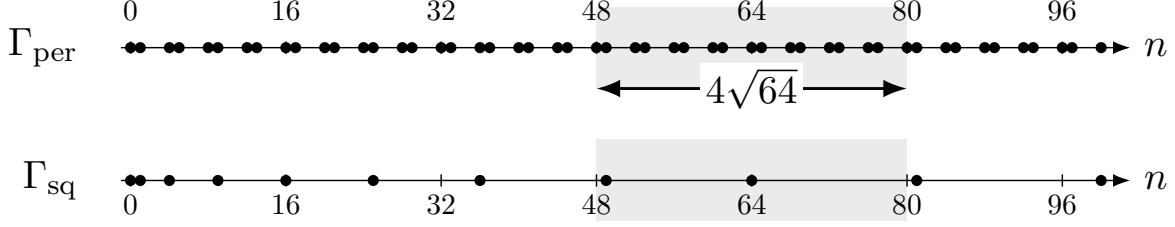

The following lemma records equivalence between various temporal thickness formulations.

\begin{lemma}
\label{lem:thickness-equivalence}
For a temporal multisequence of bounded unit-cell multiplicity, the
following are equivalent:
\begin{enumerate}[label=\textup{(\roman*)}]
\item $\mathsf S$ is Kummer-thick, i.e.,
\eqref{eq:weighted-thickness-direct} holds;
\item there are $C,c,N_0>0$ such that
\begin{equation}\label{eq:direct-temporal-count}
 \#\{k:s_k\in[N-C\sqrt N,N+C\sqrt N]\}\ge c\sqrt N,
 \qquad N\ge N_0;
\end{equation}
\item the occupied-cell set $\Gamma_{\mathsf S}$ is $\sqrt N$-thick.
\end{enumerate}
\end{lemma}

\begin{proof}
Since $\omega(s)^2=8s+4$, for $R>L+2$ the condition
$|\omega(s)-R|\le L$ is equivalent to
\[
 s\in\left[
 \frac{R^2+L^2-4}{8}-\frac{RL}{4},
 \frac{R^2+L^2-4}{8}+\frac{RL}{4}
 \right].
\]
Set $N=(R^2-4)/8$.  The interval has a center $N+L^2/8$ and
half-length $RL/4$.  Its fixed shift $L^2/8$ is absorbed by changing
the constants in a window centered at $N$, and $R\asymp\sqrt N$.
Moreover, throughout this interval,
\[
 \frac1{R+L}\le\frac1{\omega(s)}\le\frac1{R-L}.
\]
Consequently, (i) implies (ii). Conversely, the identity
\[
 |\omega(s)-\omega(N)|
 =\frac{8|s-N|}{\omega(s)+\omega(N)}
\]
shows that a temporal window $|s-N|\le C\sqrt N$ is contained
in a radial window of fixed half-length, with
$\omega(s)\asymp\sqrt N$ there.  Taking
$N=(R^2-4)/8$ for each sufficiently large $R$ proves (ii)$\Rightarrow$(i).

Each occupied unit cell contributes at least one and at most
$M_1(\mathsf S)$ temporal points. The replacement of $s$ by $\lfloor s\rfloor$
changes its distance from the center of a window by less than one.  Enlarging
the half-length by one therefore transfers the required counts in both
directions; this enlargement is absorbed in $C\sqrt N$ for large $N$.
Thus, (ii) and (iii) are equivalent.
\end{proof}

Define the equal-mass radial cells
\begin{equation}\label{eq:radial-cells}
 J_n=[\omega(n),\omega(n+1)),\qquad n\in\Nzero.
\end{equation}
They satisfy
\begin{equation}\label{eq:radial-cell-mass}
 \int_{J_n}\rho\,\dd\rho=4,
 \qquad |J_n|\asymp\omega(n)^{-1}.
\end{equation}
For $\Gamma\subset\Nzero$, put
$E_\Gamma=\bigcup_{n\in\Gamma}J_n$ and
$\dd\mu_0(\rho)=\rho\,\dd\rho$.

Because every radial cell has the same $\mu_0$-mass, the temporal counting
condition can be expressed as relative density of the union of the occupied
cells.  This reformulation is the one suited to the Fourier--Bessel
Logvinenko--Sereda theorem.

\begin{lemma}\label{lem:relative-density}
The conditions of \cref{lem:thickness-equivalence} hold if and only if
$E_{\Gamma_{\mathsf S}}$ is asymptotically $\mu_0$-relatively dense: for
some $L,\gamma,R_0>0$,
\[
 \mu_0(E_{\Gamma_{\mathsf S}}\cap[R-L,R+L])
 \ge\gamma\,\mu_0([R-L,R+L]),
 \qquad R\ge R_0.
\]
\end{lemma}

\begin{proof}
Write $\Gamma=\Gamma_{\mathsf S}$ and, for $R>L+2$, put
\[
 m_\Gamma(R,L)=\#\{n\in\Gamma:|\omega(n)-R|\le L\}.
\]
Every selected cell has $\mu_0$-mass $4$. Only the cells meeting the
two endpoints of the window can contribute a discrepancy, so
\[
 \left|\mu_0(E_\Gamma\cap[R-L,R+L])-4m_\Gamma(R,L)\right|\le8.
\]
Since $\mu_0([R-L,R+L])=2LR$, the asymptotic relative density is
equivalent, after adjusting the constants and increasing $R_0$, to
$m_\Gamma(R,L)\ge cR$ in every sufficiently distant window.
The weights $1/\omega(n)$ in that window are comparable to $1/R$.
Applying \cref{lem:thickness-equivalence} to the integer temporal set
$\Gamma$ precisely gives its $\sqrt N$-thickness, which is equivalent
to the conditions stated for $\mathsf S$.
\end{proof}

\section{A sampling theorem for Bessel waves}
\label{sec:bessel-sampling}

It remains to characterize the temporal frequency multisets for which the
associated Bessel waves form a frame.  The continuous ingredient is the
necessary-and-sufficient Fourier--Bessel Logvinenko--Sereda theorem of
Ghobber and Jaming \cite{GJ13}.  The discrete result is obtained by
sampling on the equal-$\rho\,\dd\rho$-mass cells from
\eqref{eq:radial-cells}.

For $h\in L^2(0,b)$ define
\begin{equation}
\label{eq:F-h}
 F_h(\rho)
 =
 \ip h{u_\rho}
 =
 2\int_0^b h(x)\sqrt{x}\,J_0(\rho x)\,\dd x .
\end{equation}
The order-zero Hankel Plancherel identity gives
\begin{align}
 \int_0^\infty \abs{F_h(\rho)}^2\rho\,\dd\rho
 &=4\norm h^2,
 \label{eq:plancherel}\\
 \int_0^\infty \abs{F_h'(\rho)}^2\rho\,\dd\rho
 &\leq 4b^2\norm h^2.
 \label{eq:bernstein}
\end{align}
Indeed, the first identity follows from the order-zero Hankel transform
with input $2h(x)x^{-1/2}$.  Since $J_0'=-J_1$, differentiation changes
the Bessel order from zero to one, and the order-one Hankel Plancherel
identity gives \eqref{eq:bernstein}.

\subsection{Cellwise estimates}

We first compare the radial energy on one cell with a point evaluation inside
that cell.  The shrinking diameter
$\operatorname{diam}J_n\asymp\omega(n)^{-1}$ makes the derivative error
small at high frequencies.

\begin{lemma}
\label{lem:cellwise}
There is a constant $C>0$ such that, for every sufficiently large $n$,
every $\xi\in J_n$, and every locally absolutely continuous function $F$,
\begin{align}
 \int_{J_n}\abs{F(\rho)}^2\rho\,\dd\rho
 &\leq
 8\abs{F(\xi)}^2
 +
 \frac{C}{\omega(n)^2}
 \int_{J_n}\abs{F'(\rho)}^2\rho\,\dd\rho,
 \label{eq:cell-lower}\\
 \abs{F(\xi)}^2
 &\leq
 C\int_{J_n}\abs{F(\rho)}^2\rho\,\dd\rho
 +
 \frac{C}{\omega(n)^2}
 \int_{J_n}\abs{F'(\rho)}^2\rho\,\dd\rho.
 \label{eq:cell-upper}
\end{align}
\end{lemma}

\begin{proof}
For $\rho,\xi\in J_n$, the fundamental theorem of calculus and
Cauchy--Schwarz give
\[
 \abs{F(\rho)}^2
 \leq
 2\abs{F(\xi)}^2
 +
 2\abs{\rho-\xi}
 \int_{J_n}\abs{F'(t)}^2\,\dd t.
\]
Integrating against $\rho\,\dd\rho$, using
\[
 \int_{J_n}\rho\,\dd\rho=4,
 \qquad
 \operatorname{diam}J_n\asymp\omega(n)^{-1},
\]
and observing that
\[
 \int_{J_n}\abs{F'(t)}^2\,\dd t
 \leq
 \frac{C}{\omega(n)}
 \int_{J_n}\abs{F'(t)}^2t\,\dd t,
\]
proves \eqref{eq:cell-lower}.

For \eqref{eq:cell-upper}, choose $\rho_0\in J_n$ such that
\[
 \abs{F(\rho_0)}^2
 \leq
 \frac14\int_{J_n}\abs{F(\rho)}^2\rho\,\dd\rho.
\]
Using
$F(\xi)=F(\rho_0)+\int_{\rho_0}^{\xi}F'(t)\,\dd t$
and the same Cauchy--Schwarz estimate proves \eqref{eq:cell-upper} after
enlarging $C$.
\end{proof}

The upper frame inequality follows immediately from the second cellwise
estimate and the bounded number of temporal samples in every unit cell.

\begin{lemma}
\label{lem:bessel-property}
If $M_1(\mathsf S)<\infty$, then
\(
 \{u_{\omega_k}: k\in I\}
\)
is a Bessel family in $L^2(0,b)$.
\end{lemma}

\begin{proof}
If $s_k\in[n,n+1)$, then $\omega_k\in J_n$. Thus, each $J_n$ contains at
most $M_1(\mathsf S)$ sample frequencies.  Summing
\eqref{eq:cell-upper} over the samples in each cell and then over $n$,
and using \eqref{eq:plancherel}--\eqref{eq:bernstein}, gives
\(
 \sum_{k\in I}\abs{F_h(\omega_k)}^2
 \leq C\norm h^2.
\)
The finitely many cells for which \cref{lem:cellwise} was not stated are
absorbed by increasing $C$.
\end{proof}

\subsection{Continuous Fourier--Bessel sampling}

We record the necessary-and-sufficient continuous theorem in the
normalization used in this paper. It is a rescaled order-zero case of
\cite[Theorem~4.2 and Lemma~4.1]{GJ13}.

\begin{theorem}
\label{thm:FB-LS}
Let $E\subset[0,\infty)$ be measurable.  The following are equivalent:
\begin{enumerate}[label=\textup{(\roman*)}]
\item There exist $\ell,\gamma>0$ such that
\begin{equation}
\label{eq:GJ-relative-density}
 \mu_0\bigl(E\cap[R-\ell,R+\ell]\bigr)
 \geq
 \gamma\,
 \mu_0\bigl([R-\ell,R+\ell]\bigr),
 \qquad R\geq\ell.
\end{equation}
\item There exists $c_E>0$ such that
\begin{equation}
\label{eq:GJ-restriction}
 c_E\norm h^2
 \leq
 \int_E\abs{F_h(\rho)}^2\rho\,\dd\rho,
 \qquad h\in L^2(0,b).
\end{equation}
\end{enumerate}
\end{theorem}

\begin{proof}
For completeness, in the notation of \cite{GJ13}, put
\[
 g_h(x)=2x^{-1/2}h(x)\mathbf 1_{(0,b)}(x).
\]
With their order-zero Fourier--Bessel transform $\mathcal F_0$,
\[
 \mathcal F_0g_h(y)=2\pi F_h(2\pi y).
\]
Since $\mathcal F_0$ is an involutive isometry and $g_h$ is supported in
$[0,b]$, the function $\mathcal F_0g_h$ belongs to the corresponding
Fourier--Bessel Paley--Wiener space.  The change in variables
$\rho=2\pi y$ converts their measure into a constant multiple of
$\rho\,\dd\rho$ and preserves the relative density after rescaling the window
length.  Theorem~4.2 and Lemma~4.1 of \cite{GJ13} therefore give precisely
the equivalence of \eqref{eq:GJ-relative-density} and
\eqref{eq:GJ-restriction}.
\end{proof}

The asymptotic relative-density condition in
\cref{lem:relative-density} is equivalent to
\eqref{eq:GJ-relative-density}. Indeed, if the condition in
\cref{lem:relative-density} holds with the parameters $L,\gamma,R_0$, take
\(
 \ell\geq\max\{L,R_0\}.
\)
For $R\geq\ell$, the smaller interval
$[R-L,R+L]$ lies inside $[R-\ell,R+\ell]$, and
\[
 \frac{\mu_0([R-L,R+L])}
      {\mu_0([R-\ell,R+\ell])}
 =
 \frac{L}{\ell}.
\]
Thus, \eqref{eq:GJ-relative-density} holds with density parameter
$\gamma L/\ell$.  The converse is immediate.

The following compact-remainder criterion is Peetre's lemma; see
\cite[Lemma~3]{Peetre61}.  For Hilbert spaces $X$ and $Z$, write
$\mathcal K(X,Z)$ for the space of compact linear operators from $X$ to $Z$.

\begin{lemma}
\label{lem:peetre}
Let $X,Y,Z$ be Hilbert spaces, let $B\in\mathcal B(X,Y)$, and let
$K\in\mathcal K(X,Z)$.  If
\[
 \norm x_X
 \leq
 C\norm{Bx}_Y+\norm{Kx}_Z,
 \qquad x\in X,
\]
then $\Ker B$ is finite-dimensional and $\Ran B$ is closed.  If
$\Ker B=\{0\}$, then $B$ is bounded below.
\end{lemma}

\begin{proof}
On $\Ker B$, the estimate becomes $\norm x_X\le\norm{Kx}_Z$.
Compactness of $K$ therefore makes the unit ball of $\Ker B$ compact,
so $\Ker B$ is finite-dimensional. If $B$ were not bounded below on
$(\Ker B)^\perp$, there would be unit vectors $x_j$ in that subspace
with $Bx_j\to0$. After passing to a subsequence, $Kx_j$ converges.
Applying the assumed estimate to $x_j-x_k$ shows that this subsequence
converges in $X$. Its limit is a unit vector in both $\Ker B$ and
$(\Ker B)^\perp$, a contradiction. Thus $B$ is bounded below on
$(\Ker B)^\perp$, which gives closed range and the final assertion.
\end{proof}

\subsection{The discrete characterization}

Kummer thickness and the continuous theorem first give a lower estimate
modulo a compact low-frequency term.  Cellwise discretization then yields
the corresponding semi-Fredholm estimate for the sampled Bessel waves.

\begin{proposition}
\label{prop:lower-estimate}
If \eqref{eq:weighted-thickness-direct} holds, then the analysis operator
\[
 B_{\mathsf S}:L^2(0,b)\longrightarrow\ell^2(I),
 \qquad
 B_{\mathsf S}h=(F_h(\omega_k))_{k\in I},
\]
has closed range and finite-dimensional kernel.
\end{proposition}

\begin{proof}
Put
\(
 \Gamma=\Gamma_{\mathsf S},
 \) and \(
 E=E_\Gamma.
\)
By \cref{lem:relative-density,thm:FB-LS}, there is $c_0>0$ such that
\begin{equation}
\label{eq:continuous-lower}
 c_0\norm h^2
 \leq
 \int_E\abs{F_h(\rho)}^2\rho\,\dd\rho.
\end{equation}

Choose one sample
\(
 \xi_n=\omega(s_{k(n)})\in J_n
\)
from every occupied cell $n\in\Gamma$.  Fix $N$ sufficiently large so that
\cref{lem:cellwise} applies for $n\geq N$ and
\[
 \frac{4Cb^2}{\omega(N)^2}\leq\frac{c_0}{2}.
\]
Summing \eqref{eq:cell-lower} over $n\in\Gamma$, $n\geq N$, and using
\eqref{eq:bernstein}, we obtain
\[
 \sum_{\substack{n\in\Gamma\\n\geq N}}
 \int_{J_n}\abs{F_h(\rho)}^2\rho\,\dd\rho
 \leq
 8\sum_{\substack{n\in\Gamma\\n\geq N}}
 \abs{F_h(\xi_n)}^2
 +
 \frac{c_0}{2}\norm h^2.
\]
The low-frequency restriction
\[
 K_Nh
 =
 F_h\big|_{E\cap[0,\omega(N)]}
\]
is Hilbert--Schmidt from $L^2(0,b)$ into
$L^2(E\cap[0,\omega(N)],\mu_0)$, since its kernel
$2\sqrt{x}J_0(\rho x)$ is square-integrable on the relevant bounded
rectangle.  Combining this observation with
\eqref{eq:continuous-lower} yields
\[
 \norm h^2
 \leq
 C\sum_{n\in\Gamma}\abs{F_h(\xi_n)}^2
 +
 C\norm{K_Nh}^2
 \leq
 C\sum_{k\in I}\abs{F_h(\omega_k)}^2
 +
 C\norm{K_Nh}^2.
\]
Peetre's lemma proves the assertion.
\end{proof}

The compact-remainder estimate leaves open a finite-dimensional kernel.
The following short argument rules it out.

\begin{lemma}
\label{lem:completeness}
Under \eqref{eq:weighted-thickness-direct}, the family
\(
 \{u_{\omega_k}:k\in I\}
\)
is complete in $L^2(0,b)$.
\end{lemma}

\begin{proof}
Suppose that $F_h(\omega_k)=0$ for every $k$.  Since $J_0$ is even and
entire, $F_h$ is even and entire.  The integral representation of $J_0$ and
Cauchy--Schwarz give
\(
 \abs{F_h(z)}\leq C_h e^{b\abs z}\), \(z\in\C.
\)
Thus $F_h$ has exponential type at most $b$.  Unless $F_h\equiv0$, factor
out its zero at the origin, if any, and call the remaining entire function
$G$, so that $G(0)\ne0$. Jensen's formula on the disk of radius $2R$
gives
\[
 n_G(R)\log2
 \le\frac1{2\pi}\int_0^{2\pi}\log|G(2Re^{it})|\,\dd t
       -\log|G(0)|
 \le 2bR+O(\log R),
\]
where $n_G(R)$ counts the zeros in $|z|\le R$ with multiplicity.
Thus the number of zeros of $F_h$ in that disk is $O(R)$.

Let $L,\eta,R_0$ be the constants in
\eqref{eq:weighted-thickness-direct}.  Choose
\(
 r_j=R_0+3Lj
\)
and choose $j_0$ so that $r_j-L>0$ for $j\geq j_0$.  Retain the indices
$j\geq j_0$ for which $r_j\leq R$.  The intervals
$[r_j-L,r_j+L]$ are pairwise disjoint.  In the interval centered at $r_j$,
every weight is at most $(r_j-L)^{-1}$, and hence the number of sample
occurrences in that interval is at least
\(
 \eta(r_j-L).
\)
An exact sample frequency can occur at most $M_1(\mathsf S)$ times.
Therefore, the interval contains at least a fixed positive multiple of
$r_j$ distinct zeros of $F_h$.  Since
\[
 \sum_{r_j\leq R}r_j\asymp R^2,
\]
the disk of radius $R+L$ contains at least $cR^2$ distinct zeros, a
contradiction.  Hence $F_h\equiv0$, and \eqref{eq:plancherel} gives $h=0$.
\end{proof}

The upper estimate, the closed-range conclusion, and the preceding
completeness lemma now combine to give the frame implication.

\begin{corollary}
\label{cor:sufficiency-bessel}
If \eqref{eq:weighted-thickness-direct} holds, then
\(
 \{u_{\omega_k}:k\in I\}
\)
is a frame for $L^2(0,b)$.
\end{corollary}

\begin{proof}
The Bessel property follows from \cref{lem:bessel-property}.
By \cref{prop:lower-estimate}, the analysis operator has a closed range and
finite-dimensional kernel, while \cref{lem:completeness} makes that kernel
trivial.  It is therefore bounded below.
\end{proof}

For necessity, it remains to pass from the discrete lower frame inequality
to continuous observation on the union of the occupied cells and then use
the Ghobber--Jaming theorem.

\begin{proposition}
\label{prop:necessity-bessel}
If
\(
 \{u_{\omega_k}:k\in I\}
\)
is a frame for $L^2(0,b)$, then
\eqref{eq:weighted-thickness-direct} holds.
\end{proposition}

\begin{proof}
Let $A>0$ be a lower frame bound and put
\(
 M=M_1(\mathsf S),
 \) and \(
 \Gamma=\Gamma_{\mathsf S}.
\)
For $N\in\mathbb N$, define
\[
 E_N
 =
 \bigcup_{\substack{n\in\Gamma\\n\geq N}}J_n
\]
and let
\(
 K_Nh
 =
 \bigl(F_h(\omega_k)\bigr)_{\lfloor s_k\rfloor<N}.
\)
The operator $K_N$ has a finite-dimensional range because only finitely many
samples lie in the first $N$ temporal cells.

Summing \eqref{eq:cell-upper} over all samples in cells $J_n$ with
$n\geq N$, and using the fact that each such cell contains at most $M$ sample
frequencies gives
\[
 \sum_{\lfloor s_k\rfloor\geq N}
 \abs{F_h(\omega_k)}^2
 \leq
 CM\int_{E_N}\abs{F_h(\rho)}^2\rho\,\dd\rho
 +
 \frac{CM}{\omega(N)^2}
 \int_0^\infty\abs{F_h'(\rho)}^2\rho\,\dd\rho.
\]
By \eqref{eq:bernstein}, we may choose $N$ so large that the final term is
at most $(A/2)\norm h^2$.  The lower frame inequality then gives
\begin{equation}
\label{eq:continuous-lower-mod-finite}
 \frac{A}{2}\norm h^2
 \leq
 CM\int_{E_N}\abs{F_h(\rho)}^2\rho\,\dd\rho
 +
 \norm{K_Nh}^2.
\end{equation}

Let
\(
 \mathcal R_Nh=F_h|_{E_N}.
\)
This restriction operator is bounded by \eqref{eq:plancherel}.  Moreover,
$\Ker\mathcal R_N=\{0\}$.  Indeed, the frame contains infinitely many
vectors, and bounded unit-cell multiplicity therefore forces infinitely
many occupied cells. Thus, $E_N$ contains a nondegenerate interval
$J_n$.  If $\mathcal R_Nh=0$, then the entire function $F_h$ vanishes
almost everywhere, and hence everywhere, on $J_n$.  It follows that
$F_h\equiv0$, and \eqref{eq:plancherel} gives $h=0$.

Taking square roots in \eqref{eq:continuous-lower-mod-finite} and using
$\sqrt{u+v}\leq\sqrt u+\sqrt v$ gives an estimate in the form required by
Peetre's lemma.  Since $K_N$ is of finite rank and
$\Ker\mathcal R_N=\{0\}$, the lemma shows that $\mathcal R_N$ is bounded
below:
\[
 c_N\norm h^2
 \leq
 \int_{E_N}\abs{F_h(\rho)}^2\rho\,\dd\rho,
 \qquad h\in L^2(0,b).
\]
The necessity direction of \cref{thm:FB-LS} now implies that $E_N$ is
$\mu_0$-relatively dense.  Its superset $E_\Gamma$ is therefore
asymptotically $\mu_0$-relatively dense.  By
\cref{lem:relative-density}, this is equivalent to
\eqref{eq:weighted-thickness-direct}.
\end{proof}

Together, the preceding results give the discrete characterization. Sufficiency is \cref{cor:sufficiency-bessel}, and necessity is
\cref{prop:necessity-bessel}.

\begin{theorem}
\label{thm:bessel-wave}
For a temporal multisequence of bounded unit-cell multiplicity,
\(
 \{u_{\omega(s_k)}:k\in I\}
\)
is a frame for $L^2(0,b)$ if and only if
\eqref{eq:weighted-thickness-direct} holds.
\end{theorem}

\section{Kummer-thick sampling and integer refinements}

We now transport the Bessel-wave characterization back to the atomic model.
The resulting fractional theorem is exact, while the integer specialization
retains the stronger coordinate, angle, and interpolation conclusions of
Section~2.

\begin{theorem}
\label{thm:main}
Let $a>0$ and let
$\mathsf S=(s_k)_{k\in I}\subset[0,\infty)$ have bounded unit-cell
multiplicity.  The following are equivalent:
\begin{enumerate}[label=\textup{(\roman*)}]
\item $\{T_a^{s_k}k_0:k\in I\}$ is a frame for $K_a$;
\item $\mathsf S$ is Kummer-thick;
\item the temporal count \eqref{eq:direct-temporal-count} holds;
\item the occupied-cell set $\Gamma_{\mathsf S}$ is $\sqrt N$-thick.
\end{enumerate}
The same result holds with $\theta_a$ replaced by
\[
 \theta_{a,\zeta}(z)=\exp\!\left(-a\frac{\zeta+z}{\zeta-z}\right),
 \qquad \zeta\in\Torus,
\]
with
$T_{a,\zeta}=P_{K_{\theta_{a,\zeta}}}M_z|_{K_{\theta_{a,\zeta}}}$ and the
corresponding choice of logarithmic branch.
\end{theorem}

\begin{proof}
The unitary model in \cref{prop:kummer-orbit} identifies the fractional orbit
with $\{\psi_{s_k}\}$.  The invertible transmutation in
\cref{thm:transmutation} identifies this family with Bessel waves
$\{u_{\omega(s_k)}\}$.  Apply \cref{thm:bessel-wave} and then
\cref{lem:thickness-equivalence}.

For $\zeta\in\Torus$, the unitary
\(
 (U_\zeta f)(z)=f(\overline\zeta z)
\)
maps $K_a$ onto $K_{\theta_{a,\zeta}}$, sends $k_0$ to the corresponding
kernel at zero, and satisfies
\(
 T_{a,\zeta}U_\zeta=\zeta U_\zeta T_a.
\)
If $\Log_\zeta$ is the chosen logarithm near $\zeta$, then near $w=1$,
\[
 \Log_\zeta(\zeta w)=c_\zeta+\Log w,
 \qquad e^{c_\zeta}=\zeta,\quad \Re c_\zeta=0.
\]
The holomorphic functional calculus therefore gives
\(
 T_{a,\zeta}^{\,s}U_\zeta
 =e^{sc_\zeta}U_\zeta T_a^s.
\)
The scalar $e^{sc_\zeta}$ is unimodular, so neither the frame property nor
its bounds change under rotation.
\end{proof}

Two consequences are immediate.  For a fixed temporal multisequence
satisfying the multiplicity assumption, the frame property for one atomic
mass $a>0$ is equivalent to the frame property for every $a>0$, since
the geometric criterion does not involve $a$; the frame bounds may
depend on $a$.  Also, every finite deletion from a fractional sampled
frame preserves the frame property because such a deletion leaves
all sufficiently distant radial windows unchanged.

Restricting the temporal multisequence to integer times recovers the
mesoscopic criterion and also restores the exact Hardy-coordinate
information from Proposition~2.1.

\begin{corollary}
\label{cor:integer}
For $\Gamma\subset\Nzero$,
\[
 \{P_{K_a}z^n:n\in\Gamma\}\text{ is a frame}
 \quad\Longleftrightarrow\quad
 \Gamma\text{ is }\sqrt N\text{-thick}.
\]
The optimal lower frame bound is \eqref{eq:optimal-angle}, the upper frame
bound is at most $1$, and the coefficient-interpolation criterion is given in
\cref{prop:exact-angle}.
\end{corollary}

\begin{proof}
Apply \cref{thm:main} to the integer temporal set $\mathsf S=\Gamma$ and
then use \cref{prop:exact-angle}.
\end{proof}

The bounded displacement in the radial variable preserves the criterion.
In temporal coordinates, this permits perturbations on the mesoscopic scale.

\begin{corollary}\label{cor:jitter}
Let $\mathsf S=(s_k)_{k\in I}$ and $\mathsf S'=(t_k)_{k\in I}$ be
temporal multisequences in $[0,\infty)$, each with bounded unit-cell
multiplicity.  If
\begin{equation}\label{eq:radial-jitter}
 \sup_{k\in I}|\omega(t_k)-\omega(s_k)|<\infty,
\end{equation}
then $\{T_a^{s_k}k_0\}_{k\in I}$ is a frame if and only if
$\{T_a^{t_k}k_0\}_{k\in I}$ is a frame.  In particular, the conclusion
holds if, for some $D\ge0$,
\begin{equation}\label{eq:mesoscopic-jitter}
 |t_k-s_k|\le D\sqrt{1+s_k},\qquad k\in I.
\end{equation}
\end{corollary}

\begin{proof}
Let $d\ge0$ bound the displacements in \eqref{eq:radial-jitter}.
Suppose $\mathsf S$ satisfies \eqref{eq:weighted-thickness-direct}
with constants $L,\eta,R_0$.  If $R\ge R_0$ and $R-L\ge d$, then
each index with $|\omega(s_k)-R|\le L$ satisfies
\[
 |\omega(t_k)-R|\le L+d,
 \qquad \omega(t_k)\le\omega(s_k)+d\le2\omega(s_k).
\]
Hence
\[
 \sum_{|\omega(t_k)-R|\le L+d}\frac1{\omega(t_k)}
 \ge\frac12\sum_{|\omega(s_k)-R|\le L}\frac1{\omega(s_k)}
 \ge\frac\eta2.
\]
Thus, $\mathsf S'$ is Kummer-thick.  Interchanging the two
multisequences proves the converse, and \cref{thm:main} gives the
frame equivalence.  Finally, \eqref{eq:mesoscopic-jitter} implies
\[
 |\omega(t_k)-\omega(s_k)|
 =\frac{8|t_k-s_k|}{\omega(t_k)+\omega(s_k)}
 \le\frac{8D\sqrt{1+s_k}}{\sqrt{8s_k+4}}\le4D,
\]
which proves the last assertion.
\end{proof}

The unit-cell multiplicity assumption on the perturbed multisequence is
essential to this formulation: a mesoscopic displacement bound alone can
allow increasingly many samples to coalesce in one unit cell.
The following special cases require no additional multiplicity check.

\begin{corollary}
\label{cor:affine-jitter}
Let $\Gamma\subset\Nzero$.
\begin{enumerate}[label=\textup{(\alph*)}]
\item For arbitrary $\alpha_n\in[0,1)$,
\[
 \{T_a^{n+\alpha_n}k_0:n\in\Gamma\}\text{ is a frame}
 \quad\Longleftrightarrow\quad
 \Gamma\text{ is }\sqrt N\text{-thick}.
\]
\item Fix $r>0$ and a bounded real sequence $(\delta_n)_{n\in\Gamma}$ such
that $rn+\delta_n\ge0$.  Then
\[
 \{T_a^{rn+\delta_n}k_0:n\in\Gamma\}\text{ is a frame}
 \quad\Longleftrightarrow\quad
 \Gamma\text{ is }\sqrt N\text{-thick}.
\]
\end{enumerate}
\end{corollary}

\begin{proof}
In (a), the occupied cells are exactly $\Gamma$.  In (b), the direct
temporal-count formulation in \cref{lem:thickness-equivalence} is invariant,
up to fixed changes of constants, under $n\mapsto rn+\delta_n$. The boundedness
of $(\delta_n)$ and $r>0$ also gives a uniformly bounded unit-cell
multiplicity.  Apply \cref{thm:main}.
\end{proof}

The next example emphasizes that the global natural density does not detect the
mesoscopic gaps relevant to the theorem.

\begin{example}
Set $N_j=2^{4j}$ for $j\ge1$ and retain the set
\[
 \Gamma=\Nzero\setminus
 \bigcup_{j\ge1}
 \bigl([N_j-N_j^{3/4},N_j+N_j^{3/4}]\cap\Nzero\bigr).
\]
Since $N_j^{3/4}=2^{3j}$, summing this geometric sequence shows that
the number of integers erased in $[0,X]$ is $O(X^{3/4})$.
Indeed, a contributing interval has $N_j\le2X$, because
$N_j^{3/4}\le N_j/2$. Thus, $\Gamma$ has natural density one.
However, for any fixed $C>0$, the window
$[N_j-C\sqrt{N_j},N_j+C\sqrt{N_j}]$ is completely erased for all
sufficiently large $j$, since $N_j^{3/4}/\sqrt{N_j}=2^j\to\infty$.
Consequently, the corresponding integer subsystem and every subsystem
obtained by offsets inside the retained unit cells fail to be frames.
Periodic retained sets of positive density satisfy the mesoscopic
condition, as explained in Section~5.
\end{example}

\section{Other generators}
The preceding results were formulated for the canonical generator $k_0$.
We conclude by showing that the characterization is independent of this
choice among all vectors whose integer orbit under $T_a$ is a frame.
The characterization of single-operator frame generators in
\cite{CHP20} provides the required operator.  Write
\[
 \{T_a\}'=\{X\in\mathcal B(K_a):XT_a=T_aX\}
\]
for the commutant of $T_a$, and let $GL(K_a)$ denote the group of
boundedly invertible operators on $K_a$.

\begin{proposition}
\label{prop:generator-independence}
Let $f\in K_a$ and suppose that $\{T_a^nf:n\ge0\}$ is a frame for $K_a$.  Then
there is an invertible operator $X$ that commutes with $T_a$ such that $f=Xk_0$.
Consequently, for every temporal multisequence covered by
\cref{thm:main},
\[
 \{T_a^{s_k}f\}_{k\in I}\text{ is a frame}
 \quad\Longleftrightarrow\quad
 \Gamma_{\mathsf S}\text{ is }\sqrt N\text{-thick}.
\]
\end{proposition}

\begin{proof}
The canonical orbit $\{T_a^nk_0\}_{n\geq0}$ is a Parseval frame, while
$\{T_a^nf\}_{n\geq0}$ is a frame by assumption.  By
\cite[Proposition~3.8]{CHP20}, all frame generators for a fixed bounded
operator are equivalent through its commutant. Hence, there exists
\(
   X\in\{T_a\}'\cap GL(K_a)
\)
such that
\(
   f=Xk_0.
\)
The same characterization is recorded in
\cite[Theorem~2.1]{BHKLL26}.

Since $X$ commutes with $T_a$, it commutes with the holomorphic functional
calculus of $T_a$, and therefore
\(
   XT_a^s=T_a^sX\),
   \(s\in\mathbb R.
\)
Consequently,
\(
   T_a^{s_k}f
   =
   XT_a^{s_k}k_0\),
   \(k\in I.
\)
Thus, $\{T_a^{s_k}f\}_{k\in I}$ is a frame if and only if the canonical sampled family $\{T_a^{s_k}k_0\}_{k\in I}$ is a frame.  The conclusion now
follows from Theorem~\ref{thm:main}.
\end{proof}

\begin{remark}
For the corresponding explicit $H^\infty$ parametrization of all frame
generators, including the corona condition for invertibility, see
\cite[Remark~3.9]{CHP20}.
\end{remark}

\section{Conclusion}
The exact Kummer--Bessel reduction identifies the temporal sampling
geometry of the atomic singular model.  Under bounded unit-cell
multiplicity, the weighted radial thickness characterizes the frame subsystems,
is independent of the atomic mass, and is stable under bounded radial
displacements.  The integer specialization also provides exact angle
and coefficient-interpolation criteria, while the commutant
characterization extends the result to every integer-orbit frame generator.

\section*{Declarations}

\subsection*{Funding}
Ilya Krishtal was supported in part by the Fulbright Global Scholar Award.
Javad Mashreghi was supported by the Canada Research Chairs Program
(CRC-2022-00097) and the Natural Sciences and Engineering Research Council
of Canada (Discovery Grant RGPIN-2024-04232).




\Needspace{10\baselineskip}
\subsection*{Use of generative AI}
During the preparation of this work, the authors used OpenAI's ChatGPT to
assist in manuscript organization, \LaTeX{} drafting and figures,
development and checking of mathematical arguments, and bibliography
checking and formatting. The authors reviewed and edited the AI-assisted
output and take full responsibility for the mathematical arguments,
references, and all other content of the article.

\end{document}